\documentclass[11pt,a4]{amsart}
\usepackage{amssymb,amsfonts,amsmath,amsthm,mathabx}
\usepackage{mathrsfs}
\usepackage{mathtools}
\usepackage{color}
\usepackage{ esint }
\usepackage{float} 
\usepackage{comment} 

\usepackage{biblatex}
\bibliography{reference}

\usepackage[a4paper]{geometry}

\usepackage{setspace}

\usepackage{hyperref}
\hypersetup{colorlinks=true,allcolors=black,pdfstartview=Fit,breaklinks=true}

\theoremstyle{definition}
\newtheorem{defn}{Definition}[section]
\newtheorem{rmk}{Remark}

\theoremstyle{plain}

\newtheorem{lem}[defn]{Lemma}
\newtheorem{prop}[defn]{Proposition}
\newtheorem{thm}[defn]{Theorem}

\newtheorem*{thm_no_numbering}{Theorem}

\newtheoremstyle{exp}
{\topsep}
{\topsep}
{\normalfont}
{0pt}
{}
{.}
{ }
{\thmname{#1}\thmnumber{ #2}\textnormal{\thmnote{ (#3)}}}

\theoremstyle{exp}

\newtheoremstyle{named}{}{}{\itshape}{}{\bfseries}{.}{.5em}{\thmnote{#3}#1}
\theoremstyle{named}
\newtheorem*{namedtheorem}{}

\newcommand{\C}{\mathbb{C}}
\newcommand{\R}{\mathbb{R}}
\newcommand{\Z}{\mathbb{Z}}
\newcommand{\N}{\mathbb{N}}

\newcommand{\supp}{\mathrm{supp}}
\DeclareMathOperator{\Avg}{Avg}
\numberwithin{equation}{section}

\usepackage{enumerate}

\usepackage{bm}

\usepackage{graphicx}

\begin{document}
		\title{An $H^1$ multiplier theorem on anisotropic Hardy spaces}
	\author{Yiyu Tang}
		\address{Faculty of Mathematics and Computer Science, Nicolaus Copernicus University, Chopin street 12/18, 87-100 Toruń, Poland}
	\email{ytang@mat.umk.pl}
	
	\subjclass[]{}
	\maketitle
	\begin{abstract}
		We obtain a parallel result of Sledd--Stegenga's $H^1\rightarrow L^1$ multiplier theorem under the anisotropic settings. Based on the same technique, we also prove an $H^1\rightarrow L^p$ multiplier theorem for $1\leq p<\infty$.
	\end{abstract}	
\section{Introduction}\label{Section: Introduction}
Let $H^p(\R^d)$ be the real Hardy spaces (we assume $0<p\leq1$). The following Hardy--Littlewood inequality is well known (see, for example, \cite{Stein_Harmonic_Analysis}, Chapter \RN{3}, Section 5.4)
\begin{equation*}
	\Big(\int_{\R^d}\frac{|\widehat{f}|^p}{|\xi|^{(2-p)d}}\,\mathrm{d}\xi\Big)^\frac{1}{p}
	\lesssim_{p,d}
	\|f\|_{H^p(\R^d)}.
\end{equation*}
When $p=1$, we have a stronger result, due to Sledd and Stegenga (see \cite{Sledd--Stegenga}, Theorem~1).
\begin{namedtheorem}[The $H^1$ boundedness criterion]
	 Let $\mu$ be a positive Borel measure on $\R^d\setminus\{0\}$, then
	 \begin{equation*}
	 	\int_{\R^d}|\widehat{f}(\xi)|\,\mathrm{d}\mu(\xi)	\lesssim\|f\|_{H^1(\R^d)},
	 \end{equation*} 
if and only if
	\begin{equation*}
		\sup_{\epsilon>0}\Big(\sum_{\alpha\in\Z^d\setminus\{0\}}\mu\big(Q_{\epsilon}(\alpha)\big)^2\Big)^{\frac{1}{2}}<\infty,
	\end{equation*}
	where $Q_{\epsilon}(\alpha)$ is the cube $\{(x_i)_{i=1}^d\in\R^d:\epsilon\alpha_i-\frac{\epsilon}{2}
	\leq
	x_i
	<
	\epsilon\alpha_i+\frac{\epsilon}{2}
	\}$.
\end{namedtheorem}
It is known that, under the anisotropic settings $H^p_{A}$ (we assume $0<p\leq1$), the Hardy--Littlewood inequality still holds true: Let $A$ be a dilation matrix, and $\rho_{\ast}(\cdot)$ is a quasi-norm associated with $A^\ast$, then
\begin{equation*}
	\Big(\int_{\R^d}\frac{|\widehat{f}|^p}{\rho_{\ast}(\xi)^{2-p}}\,\mathrm{d}\xi\Big)^\frac{1}{p}
	\lesssim_{p,d,A}
	\|f\|_{H^p_{A}(\R^d)}.
\end{equation*}
This inequality was proved by Bownik--Wang, see \cite{Bownik--Wang}, Corollary~8.

However, the anisotropic analog of Sledd--Stegenga's criterion seems to be absent so far. We close this gap in this note. The statement of the criterion under the anisotropic settings is more complicated and requires more preparation. Moreover, some arguments in \cite{Sledd--Stegenga} does not work any more, and we need few modifications.

Also, there are gaps in the original proofs of \cite{Sledd--Stegenga}, which either omit many non-trivial steps or are too specific (only the 1-dimensional case is proved, whereas the higher-dimensional case is not a direct result thereof). Here we complete the missing details in \cite{Sledd--Stegenga} incidentally.

\section{Acknowledgement}
This paper was completed at Nicolaus Copernicus University in Toruń, where the author was a postdoctoral researcher at that time. The author would like to thank Yuri Tomilov for introducing him to the paper \cite{Sledd--Stegenga}. The author also express the gratitude to Marcing Bownik for introducing him the theory of anisotropic Hardy spaces.

\section{Preliminaries}
In this section, most of the text is adapted from \cite{Bownik_book}.
\subsection{Dilation matrix and its homogeneous norm}
 A matrix $A\in\mathbf{GL}(d,\R)$ is a \textit{dilation matrix}, if all its eigenvalues $\{\lambda_i\}_{i=1}^d$ have modulus strictly greater than $1$. 

For each dilation matrix $A$, there exists a measurable function (not unique) $\rho_A \colon \R^d \rightarrow [0,\infty)$, called a \textit{homogenous quasi-norm associated with $A$}, satisfying the following properties:
\begin{enumerate}
	\item $\rho_A(x)=0$ if and only if $x=0$.
	\item $\rho_A(Ax)=b\rho_A(x)$ for all $x\in\R^d$, where  $b\coloneqq|\det(A)|$.
	\item $\rho_A(x+y)\leq c(\rho_A(x)+\rho_A(y))$ for all $x,y\in\R^d$, where $c>0$ is an absolute constant independent of $x$ and $y$.
\end{enumerate}
We choose $\lambda_{-}$ and $\lambda_{+}$ so that $1<\lambda_{-}<|\lambda_1|\leq\ldots\leq|\lambda_d|<\lambda_{+}$. The following growth estimate of $\rho_A$ is known.
\begin{prop}(see \cite{Lemarie-Rieusset}, Appendix B)\label{Prop: comparison on the growth of the quasi norm and Euclidean norm}
	Suppose that $\rho_A$ is a homogenous quasi-norm associated with the matrix $A$, then there is a constant $c_A>0$, so that
	\begin{equation*}
		\begin{aligned}
			c_A^{-1}\rho_{A}(x)^{\zeta_{-}}
			&\leq|x|\leq
			c_A\rho_{A}(x)^{\zeta_{+}},\text{ if }\rho_A(x)\geq1,\\
			c_A^{-1}\rho_{A}(x)^{\zeta_{+}}
			&\leq|x|\leq
			c_A\rho_{A}(x)^{\zeta_{-}},\text{ if }\rho_A(x)<1.
		\end{aligned}
	\end{equation*}
Here $\zeta_{\pm}\coloneqq \ln\lambda_{\pm}/\ln b$, and $c_A$  only depends on $\zeta_{\pm}$.
\end{prop}
The basic geometric objects in isotropic settings are Euclidean balls $B(x,r)$. For the anisotropic case, we have the following substitute. We say that $\Delta\subset\R^d$ is an \textit{ellipse}, if there exists $Q\in\mathbf{GL}(d,\R)$, so that $\Delta=Q^{-1}(B(0,1))$, where $B(0,1)$ is the (open) unit Euclidean ball in $\R^d$.
\begin{prop}(See \cite{Bownik_book}, Chapter 1, Lemma 2.2)\label{Prop: nested balls under anisotropic settings}
Suppose that $A$ is a dilation matrix, then there exists an ellipse $\Delta\subset\R^d$, so that
\begin{equation*}
	\Delta\subset A(\Delta).
\end{equation*} 
\end{prop}
This proposition is similar to the fact that $B(0,r)\subset B(0,R)$ whenever $r<R$. By rescaling, we may assume that $|\Delta|\approx1$. Although $\Delta$ is not unique, we will choose a specific one and refer it as \textit{the} ellipse associated with $A$.

The \textit{family of balls $\{\Delta_k\}_{k\in\Z}$ associated with $A$ around the origin} is
\begin{equation*}
	\Delta_{k}\coloneqq 
	A^k(\Delta),\text{ where }k\in\Z.
\end{equation*}
By Proposition~\ref{Prop: nested balls under anisotropic settings}, we have $\Delta_k\subset\Delta_{k+1}$, and $|\Delta_k|=b^k|\Delta_0|\approx b^k$.

Although all quasi-norms $\rho_A$ are equivalent (see \cite{Bownik_book}, Chapter~1, Lemma~2.4), we will use the following specific one: From the nested property of the family $\{\Delta_k\}_{k\in\Z}$, for every $x\in\R^d\setminus\{0\}$, there exists unique $k\in\Z$ so that $x\in\Delta_{k+1}\setminus\Delta_{k}$. We define
\begin{equation*}
	\rho_A(x)
	\coloneqq
	\begin{cases}
		b^k &\quad\text{if }x\in\Delta_{k+1}\setminus\Delta_{k},\\
		0 &\quad\text{otherwise}.
	\end{cases}
\end{equation*}

Let $A^\ast$ be the adjoint of $A$, since $A^\ast$ and $A$ have the same spectrum, the matrix $A^\ast$ is also a dilation matrix. We can define the quasi-norm $\rho_{A^\ast}$, the nested balls $\{\Delta^\ast_{k}\}_{k\in\Z}$, and so on.
\subsection{Anisotropic $H^1$, and the atomic decomposition}
Let $\varphi\in\mathcal{S}(\R^d)$ with $\int\varphi\neq0$. The anisotropic dilation of $\varphi$ is defined by $\varphi_k(\cdot)\coloneqq b^k\varphi(A^k\cdot)$. For $f\in L^1(\R^d)$, the smooth maximal function of $f$ with respect to $\varphi$ is
\begin{equation*}
	\mathcal{M}_{\varphi}(f)(x)
	\coloneqq
	\sup_{k\in\Z}|f\ast\varphi_k|(x).
\end{equation*}
The \textit{anisotropic} $H^1$ \textit{space associated with the dilation} $A$ is
\begin{equation*}
	H^1_{A}(\R^d)\coloneqq
	\Big\{f\in L^1(\R^d):\mathcal{M}_{\varphi}(f)\in L^1(\R^d)\Big\}
\end{equation*}

Similar to the isotropic case, the space $	H^1_{A}$ can be characterized by the atomic decomposition. We say that an $L^1$ function $a$ is an $H^1_A$-\textit{atom}, if
\begin{enumerate}
	\item There exist $x_0\in\R^d$ and $k\in\Z$, so that $\supp(a)\subset x_0+\Delta_k$.
	\item The function $a$ is bounded by $|\Delta_k|^{-1}$ almost everywhere: $\|a\|_\infty\leq|\Delta_k|^{-1}$.
	\item The function $a$ has vanishing moment: $\int a=0$.
\end{enumerate}
The \textit{atomic decomposition} of $H^1_A$ states that $f\in H^1_A$ if and only if there exist atoms $\{a_i\}_{i=1}^\infty$ and a sequence $(\lambda_i)_{i=1}^\infty\in\ell^1$, so that $	f=\sum_{i=1}^\infty\lambda_ia_i$. Moreover,
\begin{equation*}
	\begin{aligned}
		\|f\|_{H^1_A}\approx
		\inf\Big\{\sum_{i=1}^\infty|\lambda_i|:f=\sum_{i=1}^\infty\lambda_ia_i\Big\},
	\end{aligned}
\end{equation*}
where the infimum is taken over all possible atomic decomposition of $f$. For this characterization, see \cite{Bownik--Wang}, Chapter 1, Theorem 6.5.

\subsection{Duality, and anisotropic $\mathrm{BMO}$}
We denote $\mathcal{B}$ the family of translated balls associated to $A$,
\begin{equation*}
	\mathcal{B}
	\coloneqq
	\{x+\Delta_k:x\in\R^d,\text{ and }k\in\Z\}.
\end{equation*}
Assume that $1\leq q<\infty$. A locally integrable function $f$ is said to be in the  \textit{anisotropic bounded mean oscillation space} $\mathrm{BMO}_{A,q}$, if the following inequality holds:
	\begin{equation*}
		\begin{aligned}
			\sup_{B\in\mathcal{B}}\inf_{c_B\in\C}\bigg(\frac{1}{|B|}\int_{B}|f(x)-c_B|^q\,\mathrm{d}x\bigg)^{\frac{1}{q}}<\infty.
		\end{aligned}
	\end{equation*}
This quantity also defines the $\mathrm{BMO}_{A,q}$ norm. It is well-known that $\mathrm{BMO}_{A,q}$ norms are equivalent for $1\leq q<\infty$, so we will use the abbreviation $\mathrm{BMO}_{A}$ and omit $q$. Also, as in the isotropic case, the following duality relation holds:
\begin{equation*}
	(H^1_{A})^\ast
	=
	\mathrm{BMO}_A.
\end{equation*}
For basic properties of the anisotropic $\mathrm{BMO}$, see \cite{Bownik_book}, Chapter 1, Section 8.
\subsection{Use of notations and abbreviations}
For the later part of this article, we will drop the matrix $A$ and $A^\ast$ in the notations,  which does not create ambiguity. For example, we use $\rho,\rho_\ast$ to denote $\rho_{A},\rho_{A^\ast}$ respectively. Unless otherwise specified, the spaces $H^1$ and $\mathrm{BMO}$ are always interpreted as $H^1_A$ and $\mathrm{BMO}_A$.

\section{Main theorem and its proof}
To state the main theorem, we need to introduce few more notations. For $\Delta\subset\R^d$ the ellipse associated with $A$, we define $R=R(\Delta)$ be the \textit{smallest rectangle contains $\Delta$}. Here, for rectangles, we mean that the sides of $R$ are parallel to the coordinate axes. By simple convex geometry, we know that $|\Delta|\approx|R|$, and there exists $M=M_{\Delta}\in\N$, so that
\begin{equation*}
	A^{-M}(R)\subset\Delta\subset R.
\end{equation*}

For $\Delta^\ast$, similar conclusion holds. We denote the smallest rectangle containing $\Delta^\ast$ by $R^\ast$, and assume that $(A^\ast)^{-N}(R^\ast)
\subset
\Delta^\ast
\subset
 R^\ast$, where $N\in \N$ depends on the ellipse $\Delta^\ast$, and the matrix $A^\ast$.

Suppose that $R=I_1\times I_2\cdots\times I_d$. For $\alpha\in\Z^d$, We denote $R(\alpha)$ be the translation
\begin{equation*}
	\begin{aligned}
		R(\alpha)
		\coloneqq
		\Big\{(x_i)_{i=1}^d:
		|I_i|\alpha_i-\frac{|I_i|}{2}
		\leq 
		x_i
		<
		|I_i|\alpha_i+\frac{|I_i|}{2}
		\Big\},
	\end{aligned}
\end{equation*}
and $\Delta(\alpha)\subset R(\alpha)$ be the corresponding translation of $\Delta$.
	
\begin{thm}
	Let $\mu$ be a positive Borel measure on $\R^d\setminus\{0\}$, then
	\begin{equation}\label{Condition: the H^1 to F(L^1) inequality}
		\int_{\R^d}|\widehat{f}|\,\mathrm{d}\mu
		\lesssim
		\|f\|_{H^1},
	\end{equation}
if and only if
\begin{equation}\label{Condition: control of the ell^2 sum over rectangels}
\sup_{k\in\Z}\left(\sum_{\alpha\in\Z^d\setminus\{0\}}\mu\left(\left(A^\ast\right)^k\left(R^\ast\left(\alpha\right)\right)^2\right)\right)^{\frac{1}{2}}
\lesssim
1.
\end{equation}
\end{thm}
\subsection{Proof of the sufficient part}
Assume that \eqref{Condition: control of the ell^2 sum over rectangels} holds. It suffices to prove the result for $H^1$-atoms $a$. Furthermore, since we consider the magnitude of the Fourier transform, by the identity $\mathcal{F}(a(x-x_0))(\xi)=\mathrm{e}^{-2\pi\mathrm{i}x_0\cdot\xi}\widehat{a}(\xi)$, we may assume that $a$ is supported in $\Delta_k$. We decompose $\int|\widehat{a}|\,\mathrm{d}\mu$ into two parts: the integral over the region $(A^\ast)^{-(k+N)}(R^\ast)$, and over its complement.

Let us begin with the integral on $(A^\ast)^{-(k+N)}(R^\ast)$. It is easily seen that
\begin{equation*}
	\begin{aligned}
		\int_{(A^\ast)^{-(k+N)}(R^\ast)}|\widehat{a}(\xi)|\,\mathrm{d}\mu
		\lesssim
		\sum_{l=-\infty}^{-(k+N)}\int_{(A^\ast)^{l}(R^\ast)\setminus(A^\ast)^{l-N}(R^\ast)}|\widehat{a}(\xi)|\,\mathrm{d}\mu.\\
	\end{aligned}
\end{equation*}
As a consequence, to control $	\int_{(A^\ast)^{-(k+N)}(R^\ast)}|\widehat{a}(\xi)|\,\mathrm{d}\mu$, we only need to obtain good estimates of $\int_{(A^\ast)^{l}(R^\ast)\setminus(A^\ast)^{l-N}(R^\ast)}|\widehat{a}|\,\mathrm{d}\mu$, for each $l\leq-(k+N)$. 

The following finite covering property is straightforward to verify.
\begin{prop}\label{Prop: covering large rectangle by finitely many small ones}
	The region $(A^\ast)^{l}(R^\ast)\setminus(A^\ast)^{l-N}(R^\ast)$ can be covered by $\mathcal{O}_N(1)$ translations of $(A^\ast)^{l-N}(R^\ast(\alpha))$ with $\alpha\neq0$.
\end{prop}
This figure explains the proposition for $R^\ast\setminus(A^\ast)^{-N}(R^\ast)$. The rectangle drawn with bold lines is $R^\ast$. The small parallelogram centered at the origin is $(A^\ast)^{-N}(R^\ast)$, and 26 translations of it cover the region $R^\ast\setminus(A^\ast)^{-N}(R^\ast)$. The general case follows immediately by applying the mapping $(A^\ast)^{l}$.
	\begin{figure}[H]
	\centering 
	\includegraphics[width=0.4\textwidth]{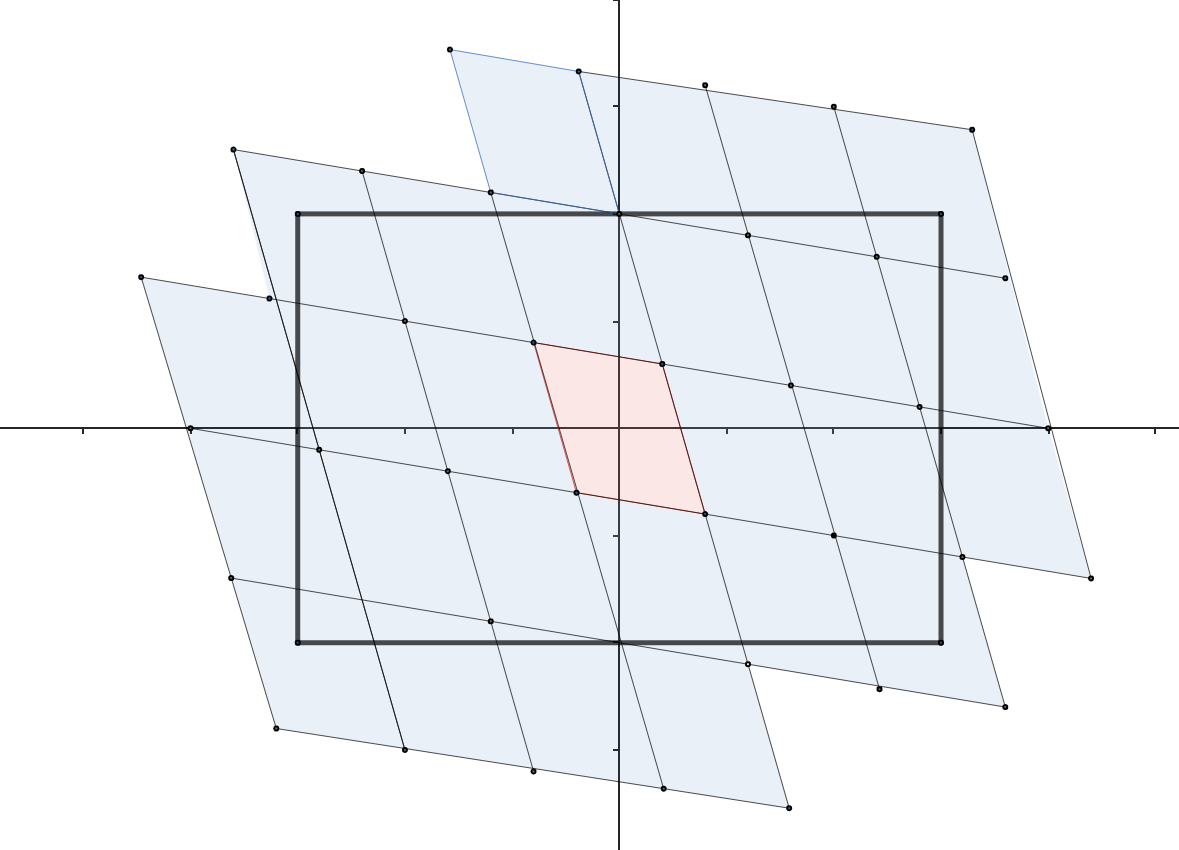} 
	\label{covering of rectangles} 
	\caption{}
\end{figure}
Using Proposition~\ref{Prop: covering large rectangle by finitely many small ones} and the Cauchy--Schwarz, it follows that
\begin{equation*}
	\begin{aligned}
		\mu\Big((A^\ast)^{l}(R^\ast)\setminus(A^\ast)^{l-N}(R^\ast)\Big)
		&\leq
		\sum_{\alpha:\,\mathcal{O}_N(1)\textrm{ terms}}
		\mu\Big((A^\ast)^{l-N}(R^\ast(\alpha))\Big)\\
		&\lesssim
		\mathcal{O}_N(1)
		\bigg(\sum_{\alpha:\,\mathcal{O}_N(1)\textrm{ terms}}
		\mu\bigg((A^\ast)^{l-N}(R^\ast(\alpha))\bigg)^2\bigg)^{\frac{1}{2}}\\
		&\lesssim
		\mathcal{O}_N(1)
		\bigg(\sum_{\alpha\neq0}
		\mu\bigg((A^\ast)^{l-N}(R^\ast(\alpha))\bigg)^2\bigg)^{\frac{1}{2}}\\
		&\lesssim_N
		1,
	\end{aligned}
\end{equation*}
where we used the hypothesis \eqref{Condition: control of the ell^2 sum over rectangels} in the last inequality.

Since $\Delta^\ast\subset R^\ast\subset (A^\ast)^{N}(\Delta^\ast)$, applying the mapping $(A^\ast)^{l-N}$ leads to
\begin{equation*}
	(A^\ast)^{l-N}(R^\ast)\subset
	(A^\ast)^{l}(\Delta^\ast)\subset
	(A^\ast)^{l}(R^\ast)\subset
	(A^\ast)^{l+N}(\Delta^\ast),
\end{equation*}
and it follows trivially that $(A^\ast)^{l}(R^\ast)\setminus(A^\ast)^{l-N}(R^\ast)
	\subset
	(A^\ast)^{l+N}(\Delta^\ast)$.
	
The analysis above shows that, if $l\leq-(k+N)$, then
\begin{equation*}
	\rho_\ast(\xi)\leq b^{l+N}\leq b^{-k},\text{ whenever }\xi\in(A^\ast)^{l}(R^\ast)\setminus(A^\ast)^{l-N}(R^\ast).
\end{equation*}
We will use the following Fourier transform estimate of $H^1$-atoms, which was established by Bownik--Wang:
\begin{prop}\label{Fourier decay of atoms}
	(See \cite{Bownik--Wang}, Lemma~5) Suppose that $a$ is an $H^1$-atom supported in $\Delta_{k}$, then
	\begin{equation*}
		|\widehat{a}(\xi)|\lesssim
		b^{k\zeta_{-}}\rho_\ast(\xi)^{\zeta_{-}},\text{ whenever }\rho_\ast(\xi)\leq b^{-k}.
	\end{equation*}
\end{prop}
Combining previous estimates,
\begin{equation*}
	\begin{aligned}
		1
		&\gtrsim
		\mu\Big((A^\ast)^{l}(R^\ast)\setminus(A^\ast)^{l-N}(R^\ast)\Big)\\
		&=
		\int_{(A^\ast)^{l}(R^\ast)\setminus(A^\ast)^{l-N}(R^\ast)}\frac{1}{\rho_{\ast}^{\zeta_{-}}(\xi)}\rho_{\ast}^{\zeta_{-}}(\xi)\,\mathrm{d}\mu\\
		&\gtrsim
		\frac{1}{b^{(l+N)\zeta_{-}}}\int_{(A^\ast)^{l}(R^\ast)\setminus(A^\ast)^{l-N}(R^\ast)}b^{-k\zeta_{-}}|\widehat{a}(\xi)|\,\mathrm{d}\mu\\
		&=
		\frac{1}{b^{(l+k+N)\zeta_{-}}}\int_{(A^\ast)^{l}(R^\ast)\setminus(A^\ast)^{l-N}(R^\ast)}|\widehat{a}(\xi)|\,\mathrm{d}\mu.\\
	\end{aligned}
\end{equation*}
Summing over $l\leq-(k+N)$ shows that
\begin{equation*}
	\begin{aligned}
		\int_{(A^\ast)^{-(k+N)}(R^\ast)}|\widehat{a}(\xi)|\,\mathrm{d}\mu
		&\lesssim
		\sum_{l=-\infty}^{-(k+N)}\int_{(A^\ast)^{l}(R^\ast)\setminus(A^\ast)^{l-N}(R^\ast)}|\widehat{a}(\xi)|\,\mathrm{d}\mu\\
		&\lesssim
		\sum_{l=-\infty}^{-(k+N)}b^{(l+k+N)\zeta_{-}}\\
		&\lesssim
		1.
	\end{aligned}
\end{equation*}
The implicit constants are independent of $k$.
\begin{rmk}
	In the original paper \cite{Sledd--Stegenga}, the estimate over $(A^\ast)^{-(k+N)}(R^\ast)$ is omitted, and the authors said it was an easy consequence of \eqref{Condition: control of the ell^2 sum over rectangels}. This is not quite accurate. In \cite{Sledd--Stegenga}, the authors write the summation \eqref{Condition: control of the ell^2 sum over rectangels} over all $\alpha\in\Z^d$, which trivially implies the estimate over $(A^\ast)^{-(k+N)}(R^\ast)$. However, the correct condition is the sum over $\alpha\in\Z^d\setminus\{0\}$.
\end{rmk}

We now turn to control the integral on the complement  $\R^d\setminus(A^\ast)^{-(k+N)}(R^\ast)$, it is easy to check that $
	\R^d\setminus R^\ast=\bigcup_{\alpha\in\Z^d\setminus\{0\}}R^\ast(\alpha)
$, applying the invertible mapping $(A^\ast)^{-(k+N)}$ gives
\begin{equation*}
	\R^d\setminus(A^\ast)^{-(k+N)}(R^\ast)=\bigcup_{\alpha\in\Z^d\setminus\{0\}}(A^\ast)^{-(k+N)}(R^\ast(\alpha)),
\end{equation*}
and the union is a disjoint union. Therefore,
\begin{equation*}
	\begin{aligned}
	\int_{\R^d\setminus(A^\ast)^{-(k+N)}(R^\ast)}|\widehat{a}|\,\mathrm{d}\mu
	&=
	\sum_{\alpha\in\Z^d\setminus\{0\}}\int_{(A^\ast)^{-(k+N)}(R^\ast(\alpha))}|\widehat{a}|\,\mathrm{d}\mu\\
	&\leq
	\sum_{\alpha\in\Z^d\setminus\{0\}}\mu\big((A^\ast)^{-(k+N)}(R^\ast(\alpha))\big)\sup_{(A^\ast)^{-(k+N)}(R^\ast(\alpha))}|\widehat{a}|.\\
	\end{aligned}
\end{equation*}
By our hypothesis \eqref{Condition: control of the ell^2 sum over rectangels},
\begin{equation*}
	\sum_{\alpha\in\Z^d\setminus\{0\}}\mu\Big((A^\ast)^{-(k+N)}(R^\ast(\alpha))\Big)^2
	\lesssim
	1.
\end{equation*}
By Cauchy--Schwarz, the proof will be completed if we can show that
\begin{equation}\label{Basic inequality: ell^2 sum of atoms over small rectangles}
	\sum_{\alpha\in\Z^d\setminus\{0\}}
	\sup_{(A^\ast)^{-(k+N)}(R^\ast(\alpha))}|\widehat{a}|^2
	\lesssim
	1.
\end{equation}

After a change of variable $\eta=(A^\ast)^{k+N}\xi$, we have
\begin{equation*}
	\sup_{\xi\in(A^\ast)^{-(k+N)}(R^\ast(\alpha))}|\widehat{a}(\xi)|^2
	=
	\sup_{\eta\in R^\ast(\alpha)}\Big|\widehat{a}\Big((A^\ast)^{-(k+N)}\eta\Big)\Big|^2.
\end{equation*}
The basic properties of the Fourier transform implies that
\begin{equation*}
	\widehat{a}\Big((A^\ast)^{-(k+N)}\eta\Big)
	=
	b^{k+N}\mathcal{F}\big(a\circ A^{k+N}\big)(\eta).
\end{equation*}
We are reduced to control the sum
\begin{equation*}
	\begin{aligned}
	\sum_{\alpha\in\Z^d\setminus\{0\}}
	\sup_{(A^\ast)^{-(k+N)}(R^\ast(\alpha))}|\widehat{a}|^2
	=
	b^{2(k+N)}\sum_{\alpha\in\Z^d\setminus\{0\}}
	\sup_{\eta\in R^\ast(\alpha)}\Big|\mathcal{F}\big(a\circ A^{k+N}\big)(\eta)\Big|^2
	\end{aligned}
\end{equation*}

\begin{prop}\label{Prop: control of the ell^2 norm over the translation of cubes}
	Assume that $d=2$. Suppose that $f$ is smooth, and $Q=I_1\times I_2$ is a rectangle in $\R^2$. The following inequality holds:
	\begin{equation*}
		\sum_{\alpha\in\Z^2}\sup_{Q(\alpha)}|f|^2
		\lesssim
		\frac{|I_1|}{|I_2|}\int|\partial_1f|^2
		+
		|Q|\int|\partial_2\partial_1f|^2
		+
		\frac{|I_2|}{|I_1|}\int|\partial_2f|^2
		+
		\frac{1}{|Q|}\int|f|^2,
	\end{equation*}
where the implicit constant in $\lesssim$ is independent of $Q$.
\end{prop}
\begin{rmk}
	This result was proved by Sledd-Stegenga for $d=1$ (see \cite{Sledd--Stegenga}, Theorem 3), which states that for any finite interval $I\subset\R$,
	\begin{equation*}
		\sum_{\alpha\in\Z}\sup_{I(\alpha)}|f|^2\leq2\bigg(\frac{1}{|I|}\int_\R|f|^2+|I|\int_\R|f^\prime|^2\bigg).
	\end{equation*}
In \cite{Sledd--Stegenga}, the authors wrote that the higher-dimensional case is somewhat complicated, so the proof was omitted. We give a proof of the two-dimensional case. For general $d\geq3$, the proof below also works, but the notations will become more annoyed.
\end{rmk}
\begin{proof}[proof of Proposition~\ref{Prop: control of the ell^2 norm over the translation of cubes}]
	Write $b_Q\coloneqq\frac{1}{|Q|}\int_{Q}f(x)\ \mathrm{d}x$, then $b_Q$ is independent of $x$, and
	\begin{equation*}
		\begin{aligned}
			\sup_{Q}|f|^2\leq2\sup_Q|f-b_Q|^2+2|b_Q|^2.
		\end{aligned}
	\end{equation*}
\textit{Control $|b_Q|^2$}: The Cauchy--Schwarz inequality implies $|b_Q|^2\leq\frac{1}{|Q|}\int_Q|f|^2$. Summing over all $\alpha\in\Z^2$, we obtain
	\begin{equation*}
		\begin{aligned}
			\sum_{\alpha\in \Z^2}|b_{Q(\alpha)}|^2
			\leq
			\sum_{\alpha\in \Z^2}\frac{1}{|Q(\alpha)|}\int_{Q(\alpha)}|f|^2
			=\frac{1}{|Q(\alpha)|}\int_{\R^2}|f|^2.
		\end{aligned}
	\end{equation*}
Here we use the fact that $|Q(\alpha)|=|Q|$ is independent of $\alpha$. 

\begin{figure}[H]
	\centering 
	\includegraphics[width=0.5\textwidth]{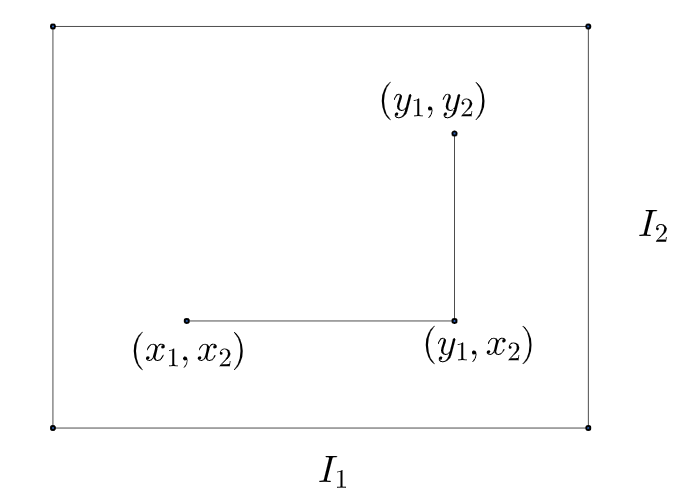} 
	\label{S-S_2} 
\end{figure}
For the term $\sup_{x\in Q}|f(x)-b_Q|^2$, we write
	\begin{equation*}
		\begin{aligned}
			|f(x)-b_Q|
			\leq
			\frac{1}{|Q|}\int_Q|f(x)-f(y)|\ \mathrm{d}y
			\leq
		\RN{1}_Q+\RN{2}_Q,
		\end{aligned}
	\end{equation*}
where we introduce the notations
	\begin{equation*}
		\begin{aligned}
			&\RN{1}_Q\coloneqq\frac{1}{|Q|}\int_{I_1}\int_{I_2}|f(x_1,x_2)-f(y_1,x_2)|\ \mathrm{d}y_2\mathrm{d}y_1,\\
			&\RN{2}_Q\coloneqq\frac{1}{|Q|}\int_{I_1}\int_{I_2}|f(y_1,x_2)-f(y_1,y_2)|\ \mathrm{d}y_2\mathrm{d}y_1.
		\end{aligned}
	\end{equation*}
Clearly, we have $\sup_{x\in Q}|f-b_Q|^2\leq 2\sup_{x\in Q}(\RN{1}_Q)^2+2\sup_{x\in Q}(\RN{2}_Q)^2$.

\textit{Control of $\sup_{x\in Q}(\RN{1}_Q)^2$}: Notice that the integrand $|f(x_1,x_2)-f(y_1,x_2)|$ is independent of $y_2$, and $|Q|=|I_1|\cdot|I_2|$, so
	\begin{equation*}
		\RN{1}_Q=\frac{1}{|I_1|}\int_{I_1}|f(x_1,x_2)-f(y_1,x_2)|\ \mathrm{d}y_1.
	\end{equation*}
	Moreover, since $x_1,y_1\in I_1$,
	\begin{equation*}
		\begin{aligned}
			|f(x_1,x_2)-f(y_1,x_2)|
			=
			\Big|\int_{x_1}^{y_1}\partial_1f(t,x_2)\ \mathrm{d}t\Big|
			\leq
			\int_{I_1}|\partial_1f|(t,x_2)\ \mathrm{d}t.
		\end{aligned}
	\end{equation*}
The integrand $\partial_1f(t,x_2)$ is independent of $y_1$, by Cauchy--Schwarz,
	\begin{equation*}
		\begin{aligned}
			\RN{1}_Q
			\leq
			\frac{1}{|I_1|}\int_{I_1}\int_{I_1}|\partial_1f|(t,x_2)\ \mathrm{d}t\mathrm{d}y_1
			\leq
			\sqrt{|I_1|}\cdot\left(\int_{I_1}|\partial_1f|^2(t,x_2)\ \mathrm{d}t\right)^{\frac{1}{2}}.
		\end{aligned}
	\end{equation*}
	Back to $\sup_{x\in Q}(\RN{1}_Q)^2$,
	\begin{equation*}
		\begin{aligned}
			\sup_{x\in Q}(\RN{1}_Q)^2
			\leq
			\sup_{\substack{x_1\in I_1\\x_2\in I_2}}|I_1|\int_{I_1}|\partial_1f|^2(t,x_2)\ \mathrm{d}t
			\leq|I_1|\int_{I_1}\sup_{\substack{x_2\in I_2}}|\partial_1f|^2(t,x_2)\ \mathrm{d}t.
		\end{aligned}
	\end{equation*}
	For $Q(\alpha)$ with $\alpha\in\Z^2$, we write it as $I_1(\alpha_1)\times I_2(\alpha_2)$. Repeating the above argument and summing over $\alpha=(\alpha_1,\alpha_2)\in\Z^2$ shows that
\begin{equation*}
	\begin{aligned}
		\sum_{\alpha\in\Z^2}\sup_{x\in Q(\alpha)}(I_{Q(\alpha)})^2
		\leq
		&		\sum_{\alpha\in\Z^2}|I_1(\alpha_1)|\int_{I_1(\alpha_1)}\sup_{\substack{x_2\in I_2(\alpha_2)}}|\partial_1f|^2(t,x_2)\,
		\mathrm{d}t\\
		=
		&\sum_{\alpha_1\in\Z}|I_1(\alpha_1)|\int_{I_1(\alpha_1)}\sum_{\alpha_2\in\Z}\sup_{\substack{x_2\in I_2(\alpha_2)}}|\partial_1f|^2(t,x_2)\,
		\mathrm{d}t\\
		=&|I_1|\int_{\R}\sum_{\alpha_2\in\Z}\sup_{\substack{x_2\in I_2(\alpha_2)}}|\partial_1f|^2(t,x_2)\, \mathrm{d}t.
	\end{aligned}
\end{equation*}
Here we use the fact that $|I_1(\alpha_1)|=|I_1|$ is independent of $\alpha_1$, and the (disjoint) union $\bigcup_{\alpha_1\in\Z}I_1(\alpha_1)$ is $\R$.

	In \cite{Sledd--Stegenga}, it was proved that for any interval $I$ and smooth $f:I\rightarrow\R$,
	\begin{equation*}
		\sup_{I}|f|^2\leq2\bigg(\frac{1}{|I|}\int_I|f|^2+|I|\int_I|f^\prime|^2\bigg).
	\end{equation*}
	Now apply this one-dimensional inequality to $I_2(\alpha_2)$ and $x_2\mapsto\partial_1f(t,x_2)$,
	\begin{equation*}
		\sup_{\substack{x_2\in I_2(\alpha_2)}}|\partial_1f|^2(t,x_2)
		\leq
		2\bigg(\frac{1}{|I_2(\alpha_2)|}\int_{I_2(\alpha_2)}|\partial_1f|^2(t,x_2)\ \mathrm{d}x_2
		+
		|I_2(\alpha_2)|\int_{I_2(\alpha_2)}|\partial_2\partial_1f|^2(t,x_2)\ \mathrm{d}x_2\bigg).
	\end{equation*}
	Summing over $\alpha_2$ gives
	\begin{equation*}
		\begin{aligned}
			\sum_{\alpha\in\Z^2}\sup_{x\in Q(\alpha)}(I_{Q(\alpha)})^2
			\leq
			\frac{2|I_1|}{|I_2|}\int_{\R}|\partial_1f|^2(t,x_2)\ \mathrm{d}x_2
			+
			2|I_1||I_2|\int_{\R}|\partial_2\partial_1f|^2(t,x_2)\ \mathrm{d}x_2.
		\end{aligned}
	\end{equation*}
	
\textit{Control of $\sup_{x\in Q}(\RN{2}_Q)^2$}: Similar to $\RN{1}_Q$,
	\begin{equation*}
		|f(y_1,x_2)-f(y_1,y_2)|=\bigg|\int_{x_2}^{y_2}\partial_2f(y_1,t)\ \mathrm{d}t\bigg|
		\leq
		\int_{I_2}|\partial_2f|(y_1,t)\ \mathrm{d}t.
	\end{equation*}
	So by Cauchy--Schwarz, the term $\RN{2}_Q$ is bounded by
	\begin{equation*}
		\begin{aligned}
			\frac{1}{|Q|}\int_{I_1}\int_{I_2}\int_{I_2}|\partial_2f|(y_1,t)\ \mathrm{d}t\mathrm{d}y_2\mathrm{d}y_1
			&=
			\frac{1}{|I_1|}\iint_{I_1\times I_2}|\partial_2f|(y_1,t)\,\mathrm{d}(y_1,t)\\
			&\leq
			\frac{\sqrt{|Q|}}{|I_1|}\left(\int_{Q}|\partial_2f|^2\right)^{1/2},
		\end{aligned}
	\end{equation*}
which trivially implies that $\sup_{x\in Q}(\RN{2}_Q)^2
		\leq
		\frac{|I_2|}{|I_1|}\int_{Q}|\partial_2f|^2$.
Summing over $\alpha$ gives
	\begin{equation*}
		\sum_{\alpha\in\Z^2}\sup_{x\in Q(\alpha)}(\RN{2}_{Q(\alpha)})^2
		\leq
		\frac{|I_2|}{|I_1|}\sum_{\alpha\in\Z^2}\int_{Q(\alpha)}|\partial_2f|^2
		=
		\frac{|I_2|}{|I_1|}\int_{\R^2}|\partial_2f|^2.
	\end{equation*}
	
	To summarize,
	\begin{equation*}
		\begin{aligned}
			\sum_{\alpha}\sup_{Q(\alpha)}|f|^2
			&\lesssim
			\sum_{\alpha}\sup_{Q(\alpha)}|f-b_{Q(\alpha)}|^2
			+
			\sum_{\alpha}|b_{Q(\alpha)}|^2\\
			&\lesssim
			\Big(\sum_{\alpha}|\RN{1}_{Q(\alpha)}|^2
			+
			\sum_{\alpha}|\RN{2}_{Q(\alpha)}|^2\Big)
			+
			\frac{1}{|Q|}\int|f|^2\\
			&\lesssim
			\frac{|I_1|}{|I_2|}\int|\partial_1f|^2
			+
			|Q|\int|\partial_2\partial_1f|^2
			+
			\frac{|I_2|}{|I_1|}\int|\partial_2f|^2
			+
			\frac{1}{|Q|}\int|f|^2.\\
		\end{aligned}
	\end{equation*}
\end{proof}
Now we back to \eqref{Basic inequality: ell^2 sum of atoms over small rectangles} and suppose that $d=2$. An application of the Proposition~\ref{Prop: control of the ell^2 norm over the translation of cubes} to the smooth function $f=\mathcal{F}(a\circ A^{k+N})$ and the rectangle $Q=R^\ast$ yields
\begin{equation*}
	\begin{aligned}
		&b^{2(k+N)}\sum_{\alpha\in\Z^2\setminus\{(0,0)\}}
		\sup_{\eta\in R^\ast(\alpha)}\Big|\mathcal{F}\big(a\circ A^{k+N}\big)(\eta)\Big|^2\\
		\lesssim
		&b^{2(k+N)}	\bigg(\frac{|I^\ast|}{|J^\ast|}\int|\partial_1f|^2
		+
		|R^\ast|\int|\partial_2\partial_1f|^2
		+
		\frac{|J^\ast|}{|I^\ast|}\int|\partial_2f|^2
		+
		\frac{1}{|R^\ast|}\int|f|^2\bigg).\\
	\end{aligned}
\end{equation*}
By hypothesis, the atom $a$ is supported in $\Delta_k=A^k(\Delta_0)$, so the function $a\circ A^{k+N}$ is supported in $A^{-N}(\Delta_0)$ and bounded by $|A^{-N}(\Delta_0)|^{-1}$. Using $\partial^{\alpha}\widehat{f}(\xi)=\mathcal{F}\big((-2\pi\mathrm{i}x)^{\alpha}f(x)\big)(\xi)$ and the Plancherel's identity,
\begin{equation*}
	\begin{aligned}
		\frac{|I^\ast|}{|J^\ast|}\int|\partial_1\mathcal{F}(a\circ A^{k+N})|^2
		&\approx	\frac{|I^\ast|}{|J^\ast|}\int_{A^{-N}(\Delta_0)}|\eta_1|^2|a (A^{k+N}\eta)|^2\,\mathrm{d}\eta_1\mathrm{d}\eta_2\\
		&\lesssim	\frac{|I^\ast|}{|J^\ast|\cdot|A^k(\Delta_0)|^2}\int_{A^{-N}(\Delta_0)}|\eta_1|^2\,\mathrm{d}\eta_1\mathrm{d}\eta_2\\
		&\lesssim_N	\frac{|I^\ast|}{|J^\ast|b^{2k}}\int_{\Delta_0}|\eta_1|^2\,\mathrm{d}\eta_1\mathrm{d}\eta_2\\
		&\lesssim_{N,R^\ast,\Delta_0}
		b^{-2k}.
	\end{aligned}
\end{equation*}
Therefore, the first term
\begin{equation*}
	b^{2(k+N)}	\frac{|I^\ast|}{|J^\ast|}\int|\partial_1f|^2
	\lesssim_{N,R^\ast,\Delta_0}
	b^{2(k+N)}b^{-2k}
	\lesssim
	1,
\end{equation*}
where implicit constants may depend on $b,N,\Delta_0$ and so on, but are independent of $k$. Other terms can be treated by the same way. Similar argument gives the higher dimensional result of \eqref{Basic inequality: ell^2 sum of atoms over small rectangles} for $d\geq3$.

\subsection{Proof of the necessary part}

We are going to use a well-known result on \textit{Bochner--Riesz multiplier}. Suppose that $\lambda>-1$. The function $m_{\lambda} \colon \R^d \rightarrow \R$ is defined by
\begin{equation*}
	m_{\lambda}(\xi) \coloneqq
	\begin{cases}
		(1-|\xi|^2)^{\lambda} &\quad\text{if }|\xi|\leq1,\\
		0 &\quad\text{if }|\xi|>1.
	\end{cases}
\end{equation*}
\begin{prop}
The (inverse) Fourier transform of $m_{\lambda}$ can be written explicitly,
\begin{equation*}
	m_{\lambda}^{\vee}(x)=
	\frac{\Gamma(\lambda+1)}{\pi^\lambda}\frac{J_{\frac{d}{2}+\lambda}(2\pi|x|)}{|x|^{\frac{d}{2}+\lambda}},
\end{equation*}
where $J_{\frac{d}{2}+\lambda}$ is the first kind Bessel function of order $\frac{d}{2}+\lambda$. Moreover, for all $x\in\R^d$, the following inequality holds:
\begin{equation*}
	m_{\lambda}^{\vee}(x)\leq \frac{C_{d,\lambda}}{1+|x|^{\frac{d+1}{2}+\lambda}},
\end{equation*}
where $C_{d,\lambda}>0$ is a finite constant independent of $x$.
\end{prop}
This result is standard and can be found in many textbooks in Fourier analysis, for example, at the Appendix B of \cite{CFA}.
\begin{lem}\label{Lemma: A special construction of H^1 function under the L^2 case}
	Let $g\in L^2(\R^d)$ and assume that $\widehat{g}=0$ on $\Delta_k^\ast$. For $\lambda>-1$, we define
	\begin{equation*}
	\chi_{k,\lambda}(x) \coloneqq
	\begin{cases}
		(1-|P(A^\ast)^{-k}x|^2)^{\lambda} &\quad\text{if }|P(A^\ast)^{-k}x|\leq1,\\
		0 &\quad\text{otherwise}.
	\end{cases}
\end{equation*}
If $f\coloneq g\cdot\widehat{\chi_{k,\lambda}}$, then for $\lambda$ sufficiently large,
	\begin{equation*}
		\|f\|_{H^1}\lesssim_{\lambda,d}
		b^{k/2}\|g\|_{L^2}.
	\end{equation*}
\end{lem}
\begin{proof}The argument here is adapted from \cite{Fefferman--Stein_real_Hardy_spaces}. We remind the reader that $\Delta^\ast_0$ is the ellipse associated with the matrix $A^\ast$, so we may assume that $\Delta^\ast_0=P^{-1}(B(0,1))$ for certain $P\in\mathbf{GL}(d,\R)$. Moreover, after re-scaling, we can assume that $|\Delta_0^\ast|\approx1$. It is trivial that $\supp(\chi_{k,\lambda})\subset\Delta^\ast_{k}$.
	
	We begin by showing that  $\int f=0$. Using the support conditions of $\widehat{g}$ and $\chi_{k,\lambda}$,
	\begin{equation*}
		\widehat{f}(0)
		=\int_{\R^d}\widehat{g}(-y)\chi_{k,\lambda}(y)\,\mathrm{d}y
		=
		\int_{\Delta^\ast_{k}}\widehat{g}(-y)\chi_{k,\lambda}(y)\,\mathrm{d}y
		=0.
	\end{equation*}
To prove that $f\in H^1$, we use duality. Let $u\in \mathrm{BMO}$, then
	\begin{equation*}
		\begin{aligned}
			\Big|\int fu\Big|
			&=\Big|\int f(u-c)\Big|
			&\leq
			\|g\|_2
			\bigg(\int|u-c|^2|\widehat{\chi_{k,\lambda}}|^2\bigg)^{\frac{1}{2}},
		\end{aligned}
	\end{equation*}
where $c$ is the average of $u$ over $A^{-k}\Delta_0$,
\begin{equation*}
	c
	=
	\Avg_{A^{-k}(\Delta_0)}
	\coloneqq
	\frac{1}{|A^{-k}\Delta_0|}\int_{A^{-k}\Delta_0}u(x)\,\mathrm{d}x.
\end{equation*}
Easy calculation shows that
\begin{equation*}
	\begin{aligned}
\widehat{\chi_{k,\lambda}}(\xi)
	=
	\frac{b^k}{|\det(P)|}
	\widehat{m_{\lambda}}(P^{-\ast}A^k\xi).
	\end{aligned}
\end{equation*}
Using the properties of $m_{\lambda}$,  we have
\begin{equation*}
	\begin{aligned}
		|\widehat{\chi_{k,\lambda}}(\xi)|
		\lesssim
		\frac{b^k}{|\det(P)|}\times
		\frac{1}{1+|P^{-\ast}A^k\xi|^{\frac{d+1}{2}+\lambda}}.
	\end{aligned}
\end{equation*}
Since all norms on $\R^d$ are equivalent, we have $|P^{-\ast}A^k\xi|\approx_P|A^k\xi|$, and now
\begin{equation*}
	\bigg(\int|u-c|^2|\widehat{\chi_{k,\lambda}}|^2\bigg)^{\frac{1}{2}}
	\lesssim_P
	b^k\bigg(\int\frac{|u-c|^2}{1+|A^kx|^{d+1+2\lambda}}
	\,\mathrm{d}x\bigg)^{\frac{1}{2}}
\end{equation*}
Decompose
\begin{equation*}
	\begin{aligned}
		\int\frac{|u-\Avg_{A^{-k}(\Delta_0)}(u)|^2}{1+|A^kx|^{d+1+2\lambda}}
		&=
		\int_{A^{-k}(\Delta_0)}\frac{|u-\Avg_{A^{-k}(\Delta_0)}(u)|^2}{1+|A^kx|^{d+1+2\lambda}}\\
		&+
		\sum_{l=0}^{\infty}
		\int_{A^{l+1-k}(\Delta_0)\setminus A^{l-k}(\Delta_0)}\frac{|u-\Avg_{A^{-k}(\Delta_0)}(u)|^2}{1+|A^kx|^{d+1+2\lambda}}.
	\end{aligned}
\end{equation*}
The trivial fact $|A^kx|\geq0$ implies that
\begin{equation*}
	\begin{aligned}
	\int_{A^{-k}(\Delta_0)}\frac{|u-\Avg_{A^{-k}(\Delta_0)}(u)|^2}{1+|A^kx|^{d+1+2\lambda}}
	&\lesssim
	\int_{A^{-k}(\Delta_0)}|u-\Avg_{A^{-k}(\Delta_0)}(u)|^2\\
	&\lesssim
	|A^{-k}(\Delta_0)|\cdot\|u\|_{\mathrm{BMO}}^2\\
	&\approx
	b^{-k}\cdot\|u\|_{\mathrm{BMO}}^2.
	\end{aligned}
\end{equation*}
We decompose, for each $l\geq0$, the integral $	\int_{A^{l+1-k}(\Delta_0)\setminus A^{l-k}(\Delta_0)}$ as
\begin{equation*}
	\begin{aligned}
		&\int_{A^{l+1-k}(\Delta_0)\setminus A^{l-k}(\Delta_0)}\frac{|u-\Avg_{A^{l+1-k}(\Delta_0)}(u)|^2}{1+|A^kx|^{d+1+2\lambda}}\\
		+
		&\int_{A^{l+1-k}(\Delta_0)\setminus A^{l-k}(\Delta_0)}\frac{|\Avg_{A^{l+1-k}(\Delta_0)}(u)-\Avg_{A^{-k}(\Delta_0)}(u)|^2}{1+|A^kx|^{d+1+2\lambda}}.
	\end{aligned}
\end{equation*}
For $x\in A^{l+1-k}(\Delta_0)\setminus A^{l-k}(\Delta_0)$,
\begin{equation*}
	A^kx\in A^{l+1}(\Delta_0)\setminus A^{l}(\Delta_0), \text{ and } \rho(A^kx)=b^{l}\geq1.
\end{equation*} 
Proposition~\ref{Prop: comparison on the growth of the quasi norm and Euclidean norm} says that $|x|\gtrsim\rho(x)^{\zeta_-}$ whenever $\rho(x)\geq1$. Replacing $x$ by $A^kx$, we get $|A^kx|
	\gtrsim
	\rho(A^kx)^{\zeta_-}
	= 
	b^{l\zeta_-}$.
	
Now
\begin{equation*}
	\begin{aligned}
		\int_{A^{l+1-k}(\Delta_0)\setminus A^{l-k}(\Delta_0)}\frac{|u-\Avg_{A^{l+1-k}(\Delta_0)}(u)|^2}{1+|A^kx|^{d+1+2\lambda}}
		&\lesssim
		\int_{A^{l+1-k}(\Delta_0)}\frac{|u-\Avg_{A^{l+1-k}(\Delta_0)}(u)|^2}{b^{l(d+1+2\lambda)\zeta_{-}}}\\
		&\lesssim
		\frac{|A^{l+1-k}(\Delta_0)|}{b^{l(d+1+2\lambda)\zeta_{-}}}\|u\|_{\mathrm{BMO}}^2\\
		&\approx
		\frac{b^{l+1-k}}{b^{l(d+1+2\lambda)\zeta_{-}}}\|u\|_{\mathrm{BMO}}^2.
	\end{aligned}
\end{equation*}
For $u\in\mathrm{BMO}$, using the facts $\Delta_0\subset A(\Delta_0)$ and $|A(\Delta_0)|=b|\Delta_0|$,
\begin{equation*}
	\begin{aligned}
		|\Avg_{\Delta_0}(u)-\Avg_{A(\Delta_0)}(u)|
		&=
		\frac{1}{|\Delta_0|}\Big|\int_{\Delta_0}u-\Avg_{A(\Delta_0)}(u)\Big|\\
		&\lesssim
		\frac{b}{|A(\Delta_0)|}\int_{A(\Delta_0)}\Big|u-\Avg_{A(\Delta_0)}(u)\Big|\\
		&\leq
		b
		\|u\|_{\mathrm{BMO}}.
	\end{aligned}
\end{equation*}
By telescoping,
\begin{equation*}
	|\Avg_{A^{l+1-k}(\Delta_0)}(u)-\Avg_{A^{-k}(\Delta_0)}(u)|
	\lesssim
	(l+1)b\|u\|_{\mathrm{BMO}}.
\end{equation*}
Consequently,
\begin{equation*}
	\begin{aligned}
	&\int_{A^{l+1-k}(\Delta_0)\setminus A^{l-k}(\Delta_0)}\frac{|\Avg_{A^{l+1-k}(\Delta_0)}(u)-\Avg_{A^{-k}(\Delta_0)}(u)|^2}{1+|A^kx|^{d+1+2\lambda}}\\
	\lesssim
	&(l+1)^2b^2\|u\|_{\mathrm{BMO}}^2\int_{A^{l+1-k}(\Delta_0)\setminus
	A^{l-k}(\Delta_0)}\frac{1}{1+|A^kx|^{d+1+2\lambda}}\\
	\lesssim
	&(l+1)^2b^2\|u\|_{\mathrm{BMO}}^2\frac{b^{l+1-k}}{b^{l(d+1+2\lambda)\zeta_{-}}}.\\
	\end{aligned}
\end{equation*}
Summing over $l\in\N$,
\begin{equation*}
	\begin{aligned}
		&\sum_{l=0}^{\infty}
		\int_{A^{l+1-k}(\Delta_0)\setminus A^{l-k}(\Delta_0)}\frac{|u-\Avg_{A^{-k}(\Delta_0)}(u)|^2}{1+|A^kx|^{d+1+2\lambda}}\\
		\lesssim
		&\sum_{l=0}^{\infty}
		\frac{b^{l+1-k}(1+(l+1)^2b^2)}{b^{l(d+1+2\lambda)\zeta_{-}}}\|u\|^2_{\mathrm{BMO}}.\\
	\end{aligned}
\end{equation*}
Now we choose $\lambda$ large, so that $1<(d+1+2\lambda)\zeta_{-}$, then the sum is controlled by $b^{-k}\|u\|^2_{\mathrm{BMO}}$. Combining these estimates gives us
\begin{equation*}
	\begin{aligned}
		\Big|\int fu\Big|\lesssim
		b^{\frac{k}{2}}\|g\|_2\|u\|_{\mathrm{BMO}}.
	\end{aligned}
\end{equation*}
By duality, the $H^1$ norm of $f$ is controlled by $\mathcal{O}(b^{\frac{k}{2}}\|g\|_2)$.
\end{proof}

\begin{rmk}
Under the isotropic settings, Sledd \& Stegenga proved results in \cite{Sledd--Stegenga} by choosing the ball multiplier $m_0$. However, due to the anisotropic nature, the ball multiplier does not work here: We need an additional $\lambda>0$ to offset the increase of $b^l$ in the numerator. Readers can easily verify that, if $\lambda=0$, then the key inequality $1<(d+1+2\lambda)\zeta_{-}$ does not necessarily hold for a general dilation matrix $A$.
\end{rmk}

Back to the proof of the necessary part. Suppose that \eqref{Condition: the H^1 to F(L^1) inequality} holds. Applying \eqref{Condition: the H^1 to F(L^1) inequality} for $f=g\cdot\widehat{\chi_{k,\lambda}}$ with $\|g\|_2\leq1$,
\begin{equation*}
	\Big|\int\widehat{g}\ast\chi_{k,\lambda}\,\mathrm{d}\mu\Big|
	=
	\int|\widehat{f}|\,\mathrm{d}\mu\lesssim\|f\|_{H^1}\lesssim b^{\frac{k}{2}},
\end{equation*}
while
\begin{equation*}
	\int\widehat{g}\ast\chi_{k,\lambda}\,\mathrm{d}\mu
	=
	\int\widehat{g}(y)\int\chi_{k,\lambda}(x-y)\,\mathrm{d}\mu(x)\mathrm{d}y.
\end{equation*}
Therefore,
\begin{equation*}
	\Big|\int\widehat{g}(y)\int\chi_{k,\lambda}(x-y)\,\mathrm{d}\mu(x)\mathrm{d}y\Big|\lesssim b^{\frac{k}{2}}
\end{equation*}
By our hypothesis, the Fourier transform $\widehat{g}$ is supported in $(\Delta^\ast_k)^c$, so by duality,
\begin{equation*}
	\begin{aligned}
	\int_{(\Delta^\ast_k)^c}\bigg|\int\chi_{k,\lambda}(x-y)\,\mathrm{d}\mu(x)\bigg|^2\,\mathrm{d}y
	\lesssim
	b^k.
	\end{aligned}
\end{equation*}
Since $\supp(\chi_{k,\lambda})\subset\Delta^\ast_k$, we have $	\int\chi_{k,\lambda}(x-y)\,\mathrm{d}\mu(x)=\int_{y+\Delta^\ast_k}\chi_{k,\lambda}(x-y)\,\mathrm{d}\mu(x)$.
For $\epsilon>0$ small, and
$x\in(1-\epsilon)\Delta^\ast_k$, it follows immediately that
\begin{equation*}
	P(A^\ast)^{-k}x
	\in
	(1-\epsilon)P(A^\ast)^{-k}\Delta^\ast_k
	=
	(1-\epsilon)B(0,1),
\end{equation*}
In consequence, the norm $|P(A^\ast)^{-k}x|\leq1-\epsilon$, and $\chi_{k,\lambda}(x)\gtrsim_{\epsilon,\lambda}1$. Thus,
\begin{equation*}
	\begin{aligned}
	\int_{y+\Delta^\ast_k}\chi_{k,\lambda}(x-y)\,\mathrm{d}\mu(x)
	&\geq
	\int_{y+(1-\epsilon)\Delta^\ast_k}\chi_{k,\lambda}(x-y)\,\mathrm{d}\mu(x)\\
	&\gtrsim_{\epsilon,\lambda}
	\mu\big(y+(1-\epsilon)\Delta^\ast_k\big).
	\end{aligned}
\end{equation*}
This amounts to the following inequality:
\begin{equation*}
	\begin{aligned}
		\int_{(\Delta^\ast_k)^c}\mu\big(y+(1-\epsilon)\Delta^\ast_k\big)^2\,\mathrm{d}y
		\lesssim
		b^k.
	\end{aligned}
\end{equation*}
We know that $\Delta^\ast_k\subset (A^\ast)^k(R^\ast)$, so
\begin{equation*}
 \bigcup_{\alpha\neq0}(A^\ast)^k(R^\ast(\alpha))
 =
 \Big((A^\ast)^k(R^\ast)\Big)^c
 \subset 	
 (\Delta^\ast_k)^c,
\end{equation*}
and
\begin{equation*}
	\begin{aligned}
		\sum_{\alpha\neq0}\int_{(A^\ast)^k(R^\ast(\alpha))}\mu\big(y+(1-\epsilon)\Delta^\ast_k\big)^2\,\mathrm{d}y
		\lesssim
		b^k.
	\end{aligned}
\end{equation*}
Also, since $(A^\ast)^{-N}(R^\ast)\subset\Delta^\ast_0$,
\begin{equation*}
(1-\epsilon)(A^\ast)^{-N}(A^\ast)^k(R^\ast)
\subset
(1-\epsilon)\Delta^\ast_k.
\end{equation*}
By the geometry of $R^\ast$, it is easy to see that there exist a small constant $c>0$, so that 
\begin{equation*}
	c\cdot R^\ast\subset(1-\epsilon)(A^\ast)^{-N}(R^\ast).
\end{equation*}
Replacing  $(1-\epsilon)\Delta^\ast_k$ by the smaller subset $(A^\ast)^k(cR^\ast)$ leads to
\begin{equation}\label{Control the sum of measures on translate rectangels}
	\begin{aligned}
		\sum_{\alpha\neq0}\int_{(A^\ast)^k(R^\ast(\alpha))}\mu\big(y+(A^\ast)^k(cR^\ast)\big)^2\,\mathrm{d}y
		\lesssim
		b^k.
	\end{aligned}
\end{equation}
Next, we take a specific value $c=1/2$ as an example to explain the proof, the general case $c>0$ is similar, but the notations will be more complicated. 

For $\alpha$ fixed, we decompose $R^\ast(\alpha)$ into 9 congruent parts $r_i$. It is evident that, the translated rectangle $y+\frac{1}{2}R^\ast$ contains $r_i$, whenever $y$ is near the center of $r_i$. Thus, we define a smaller rectangle concentric with $r_i$:
\begin{equation*}
	r_{i}^\prime
	\coloneqq
	\big\{y\in r_i:\textrm{ the rectangle }y+\frac{1}{2}R^\ast\textrm{ contains }r_i\big\}.
\end{equation*}
A trivial verification shows that $|r_i|/|r_{i}^\prime|\approx1$. For general $c>0$, the number of congruent parts into which $R^\ast(\alpha)$ is divided, and the area quotient $|r_i|/|r_{i}^\prime|$ depend of course on $c$, but the important point is that they are independent of $k$.
	\begin{figure}[H]
	\centering 
	\includegraphics[width=0.6\textwidth]{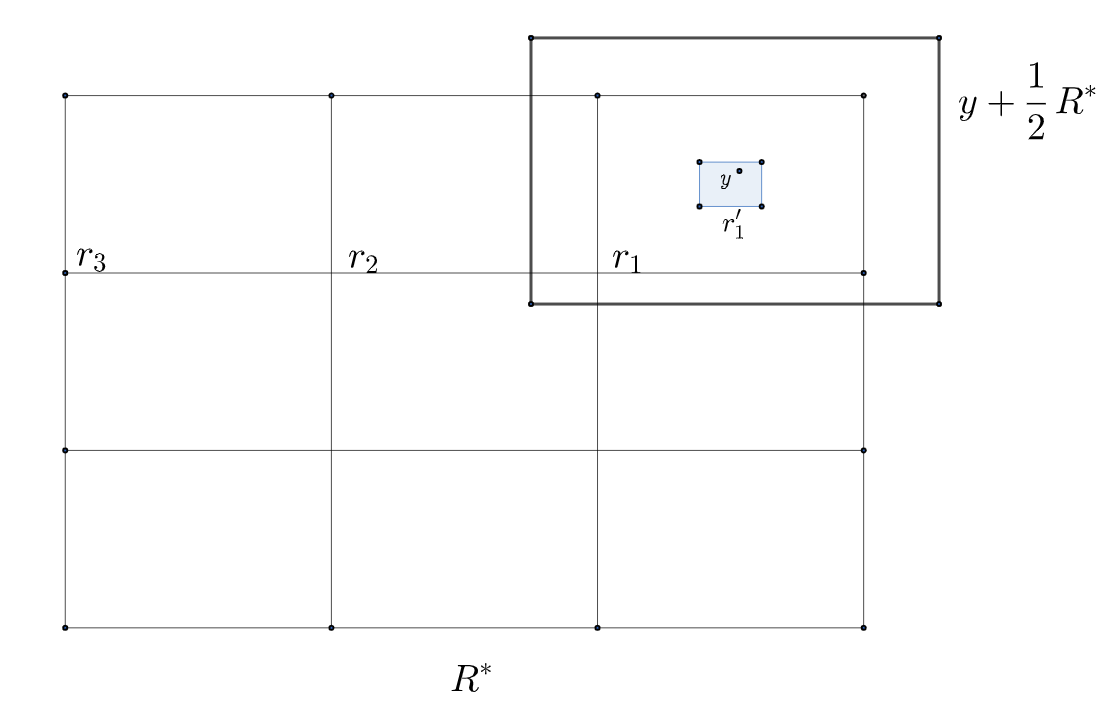} 
	\label{covering of rectangles by its half} 
	\caption{The rectangle $y+\frac{1}{2}R^\ast$ contains $r_1$, whenever $y\in r_1^\prime$.}
\end{figure}
We have  $r_i\subset y+\frac{1}{2}R^\ast,\text{ when }y\in r^\prime_i$. Applying the mapping $(A^\ast)^k$ shows that
\begin{equation*}
	(A^\ast)^k(r_i)
	\subset y+(A^\ast)^k\Big(\frac{1}{2}R^\ast\Big),\text{ whenever }y\in(A^\ast)^k(r^\prime_i).
\end{equation*}
Based on the previous analysis, the following inequality is clearly established:
\begin{equation*}
	\begin{aligned}
		\int_{(A^\ast)^k(R^\ast(\alpha))}\mu\big(y+(A^\ast)^k(cR^\ast)\big)^2\,\mathrm{d}y
		&=
		\sum_{1\leq i\leq9}
		\int_{(A^\ast)^k(r_i)}\mu\big(y+(A^\ast)^k(cR^\ast)\big)^2\,\mathrm{d}y\\
		&\geq
		\sum_{1\leq i\leq9}
		\int_{(A^\ast)^k(r^\prime_i)}\mu\big(y+(A^\ast)^k(cR^\ast)\big)^2\,\mathrm{d}y\\
		&\geq
		\sum_{1\leq i\leq9}
		\int_{(A^\ast)^k(r^\prime_i)}\mu\big((A^\ast)^k(r_i)\big)^2\,\mathrm{d}y\\
		&=
		\sum_{1\leq i\leq9}
		\mu\big((A^\ast)^k(r_i)\big)^2|(A^\ast)^k(r^\prime_i)|.\\
\end{aligned}
\end{equation*}
Since $|r^\prime_i|\approx|r_i|=|R^\ast|/9$, and $|(A^\ast)^k(r^\prime_i)|=b^k|r^\prime_i|$, the above inequality can be further written as
\begin{equation*}
	\begin{aligned}
		\int_{(A^\ast)^k(R^\ast(\alpha))}\mu\big(y+(A^\ast)^k(cR^\ast)\big)^2\,\mathrm{d}y
		&\gtrsim
		b^k|R^\ast|\sum_{1\leq i\leq9}
		\mu\big((A^\ast)^k(r_i)\big)^2\\
		&\gtrsim
		b^k|R^\ast|
		\mu\Big(\bigcup_{1\leq1\leq9}(A^\ast)^k(r_i)\Big)^2.\\
	\end{aligned}
\end{equation*}
The last line comes from the trivial estimate $ \sum_{1\leq i\leq 9}|a_i|^2\geq \frac{1}{9}\big(\sum_{1\leq i\leq 9}|a_i|\big)^2$.

It is clear that
\begin{equation*}
	\bigcup_{1\leq1\leq9}(A^\ast)^k(r_i)
	=
	(A^\ast)^k\Big(\bigcup_{1\leq1\leq9}r_i\Big)
	=
	(A^\ast)^k(R^\ast(\alpha)).
\end{equation*}
Combining the above estimates gives us
\begin{equation*}
	\begin{aligned}
		\int_{(A^\ast)^k(R^\ast(\alpha))}\mu\big(y+(A^\ast)^k(cR^\ast)\big)^2\,\mathrm{d}y
		&\gtrsim
		b^k|R^\ast|
		\mu\bigg((A^\ast)^k(R^\ast(\alpha))\bigg)^2.\\
	\end{aligned}
\end{equation*}
Now back to \eqref{Control the sum of measures on translate rectangels}, the following inequality holds,
\begin{equation*}
	\begin{aligned}
		b^k
		&\gtrsim
		\sum_{\alpha\neq0}\int_{(A^\ast)^k(R^\ast(\alpha))}\mu\big(y+(A^\ast)^k(cR^\ast)\big)^2\,\mathrm{d}y\\
		&\gtrsim
		\sum_{\alpha\neq0}b^k|R^\ast|
		\mu\bigg((A^\ast)^k(R^\ast(\alpha))\bigg)^2.\\
	\end{aligned}
\end{equation*}
This implies that, for all $k\in\Z$,
\begin{equation*}
	\begin{aligned}
		\sum_{\alpha\neq0}
		\mu\Big((A^\ast)^k\big(R^\ast(\alpha)\big)\Big)^2
		\lesssim
		|R^\ast|^{-1}
		\approx
		1,
	\end{aligned}
\end{equation*}
which is exactly our desired result.
\begin{rmk}
 In the original proof in \cite{Sledd--Stegenga}, from the inequality
	\begin{equation*}
		\Big|\int\widehat{g}(y)\int\chi_{k,\lambda}(x-y)\,\mathrm{d}\mu(x)\mathrm{d}y\Big|\lesssim b^{\frac{k}{2}},
	\end{equation*}
the authors directly concluded that
	\begin{equation*}
		\begin{aligned}
			\int_{\R^d}\bigg|\int\chi_{k,\lambda}(x-y)\,\mathrm{d}\mu(x)\bigg|^2\,\mathrm{d}y
			\lesssim
			b^k,
		\end{aligned}
	\end{equation*}
and then \eqref{Condition: control of the ell^2 sum over rectangels} can be easily deduced. 

Of course, this is inaccurate. Since $\widehat{g}=0$ in $\Delta_k^\ast$, we have no information on the behavior of $y\mapsto\int\chi_{k,\lambda}(x-y)\,\mathrm{d}\mu(x)$ for $y\in \Delta_k^\ast$, and we can only choose $\Omega=(\Delta_k^\ast)^c$ (rather than $\R^d$) when using the duality:
\begin{equation*}
	\|f\|_{L^2(\Omega)}
	=
	\sup_{\|g\|_{L^2(\Omega)}\leq1}\left|\int fg\right|.
\end{equation*}
\end{rmk}
\section{Further results} Under the isotropic settings, the following theorem was announced in \cite{Sledd--Stegenga}.
\begin{namedtheorem}[The $H^1\mapsto L^p$ boundedness criterion]
	Let $\mu$ be a positive Borel measure on $\R^d\setminus\{0\}$, and $H^1(\R^d)$ be the isotropic Hardy space, then
	\begin{equation*}
		\left(\int_{\R^d}|\widehat{f}(\xi)|^p\,\mathrm{d}\mu(\xi)\right)^{\frac{1}{p}}	\lesssim\|f\|_{H^1(\R^d)},
	\end{equation*} 
	if and only if the following inequalities hold respectively:
	\begin{enumerate}
		\item When $1\leq p<2$,	\begin{equation*}
			\sup_{\epsilon>0}\Big(\sum_{\alpha\in\Z^d\setminus\{0\}}\mu\big(Q_{\epsilon}(\alpha)\big)^\frac{2}{2-p}\Big)^{\frac{2-p}{2}}<\infty.
		\end{equation*}
		\item When $2\leq p<\infty$,
		\begin{equation*}
			\sup_{\epsilon>0}\mu\Big(\{x\in\R^d:\epsilon\leq|x|<2\epsilon\}\Big)<\infty.
		\end{equation*}
	\end{enumerate}
\end{namedtheorem}
This theorem is stated in \cite{Sledd--Stegenga} but without proofs. We feel that, although the authors of \cite{Sledd--Stegenga} wrote that this was just an easy generalization of \textbf{the} $H^1$ \textbf{boundedness criterion} in Section~\ref{Section: Introduction}, some modifications are still necessary and not that trivial. In this subsection, we state and give a sketchy proof of the corresponding result under the anisotropic settings.
\begin{thm}\label{The H^1 to L^p boundedness criterion}
	Let $\mu$ be a positive Borel measure on $\R^d\setminus\{0\}$. Then
	\begin{equation}\label{Condition: the H^1 to F(L^p) inequality}
			\left(\int_{\R^d}|\widehat{f}(\xi)|^p\,\mathrm{d}\mu(\xi)\right)^{\frac{1}{p}}\lesssim
			\|f\|_{H^1(\R^d)},
	\end{equation}
	if and only if the following inequalities hold respectively.
	\begin{enumerate}
		\item When $1\leq p<2$,	\begin{equation}\label{Condition: control of the ell^{2/(2-p)} sum over rectangels}
			\sup_{k\in\Z}\left(\sum_{\alpha\in\Z^d\setminus\{0\}}\mu\left(\left(A^\ast\right)^k\left(R^\ast\left(\alpha\right)\right)^\frac{2}{2-p}\right)\right)^{\frac{2-p}{2}}
			<
			\infty.
		\end{equation}
	\item When $2\leq p<\infty$,
	\begin{equation}\label{Condition: control of the ell^{infty} sum over annulus}
		\sup_{k\in\Z}\mu\Big(\{x\in\R^d:\rho_\ast(x)=b^k\}\Big)
		<
		\infty.
	\end{equation}
	\end{enumerate}
\end{thm}

\subsection{Proof of Theorem~\ref{The H^1 to L^p boundedness criterion}, when $1\leq p<2$} 
The proof at is somewhat similar to that of $p=1$.

\paragraph{\textit{The sufficient part}} We first assume that \eqref{Condition: control of the ell^{2/(2-p)} sum over rectangels} holds, it suffices to prove that for all $H^1$-atoms $a$,
\begin{equation*}
	\int|\widehat{a}|^p\,\mathrm{d}\mu
	\lesssim
	1.
\end{equation*}
Same as before, we assume that $a$ is supported in $\Delta_k$, and we  decompose the integral into two parts: integral over the region $(A^\ast)^{-(k+N)}(R^\ast)$, and over its complement.

 Using Proposition~\ref{Prop: covering large rectangle by finitely many small ones} and H\"older's inequality, for $k\leq-(l+N)$, 
\begin{equation*}
	\begin{aligned}
		\mu\Big((A^\ast)^{l}(R^\ast)\setminus(A^\ast)^{l-N}(R^\ast)\Big)
		&\leq
		\sum_{\alpha:\,\mathcal{O}_N(1)\textrm{ terms}}
		\mu\Big((A^\ast)^{l-N}(R^\ast(\alpha))\Big)\\
		&\lesssim
		\mathcal{O}_{N,p}(1)
		\bigg(\sum_{\alpha:\,\mathcal{O}_N(1)\textrm{ terms}}
		\mu\bigg((A^\ast)^{l-N}(R^\ast(\alpha))\bigg)^\frac{2}{2-p}\bigg)^{\frac{2-p}{2}}\\
		&\lesssim_{N,p}
		1.
	\end{aligned}
\end{equation*}
Now using Proposition~\ref{Fourier decay of atoms}, it is easy to prove that
\begin{equation*}
	\begin{aligned}
		1
		&\gtrsim
		\int_{(A^\ast)^{l}(R^\ast)\setminus(A^\ast)^{l-N}(R^\ast)}\frac{1}{\rho_{\ast}^{p\zeta_{-}}(\xi)}\rho_{\ast}^{p\zeta_{-}}(\xi)\,\mathrm{d}\mu\\
		&\gtrsim
		\frac{1}{b^{(l+k+N)p\zeta_{-}}}\int_{(A^\ast)^{l}(R^\ast)\setminus(A^\ast)^{l-N}(R^\ast)}|\widehat{a}(\xi)|^p\,\mathrm{d}\mu.
	\end{aligned}
\end{equation*} 
Summing over $-\infty<l\leq-(k+N)$ yields that $	\int_{(A^\ast)^{-(k+N)}(R^\ast)}|\widehat{a}|^p\,\mathrm{d}\mu
\lesssim
1$.

For the integral over $\R^d\setminus(A^\ast)^{-(k+N)}(R^\ast)$, we just notice that
\begin{equation*}
	\begin{aligned}
		\int_{\R^d\setminus(A^\ast)^{-(k+N)}(R^\ast)}|\widehat{a}|^p\,\mathrm{d}\mu
		&\leq
		\sum_{\alpha\in\Z^d\setminus\{(0,0)\}}\mu\big((A^\ast)^{-(k+N)}(R^\ast(\alpha))\big)\sup_{(A^\ast)^{-(k+N)}(R^\ast(\alpha))}|\widehat{a}|^p.
	\end{aligned}
\end{equation*}
Using \eqref{Basic inequality: ell^2 sum of atoms over small rectangles}, applying the H\"older's inequality for the pair $(\frac{2}{2-p},\frac{2}{p})$ gives the desired result.

\paragraph{\textit{The necessary part}} Now, on the contrary, we suppose that \eqref{Condition: the H^1 to F(L^p) inequality} holds. With a little effort, the following variant of Lemma~\ref{Lemma: A special construction of H^1 function under the L^2 case} is not difficult to prove.
\begin{lem}\label{Lemma: A special construction of H^1 function under the L^r case}
	Suppose that $1<r<\infty$. Let $g\in L^r(\R^d)$ and assume that $\widehat{g}=0$ on $\Delta_k^\ast$. If $f\coloneq g\cdot\widehat{\chi_{k,\lambda}}$, then for $\lambda$ sufficiently large,
	\begin{equation*}
		\|f\|_{H^1}
		\lesssim_{\lambda,d,r}
		b^{k/r}\|g\|_{L^r}.
	\end{equation*}
\end{lem}
We omit the proof of this lemma since it is a step-by-step modification of that of Lemma~\ref{Lemma: A special construction of H^1 function under the L^2 case}. Regarding the meaning of “sufficiently large”, a careful calculation shows that any $\lambda$ strictly greater than $\frac{1}{r^\prime\zeta_{-}}-\frac{d+1}{2}$ works, where $r^\prime\coloneqq r/(r-1)$ is the conjugate exponent of $r$.

We apply $f=g\cdot\widehat{\chi_{k,\lambda}}$ to \eqref{Condition: the H^1 to F(L^p) inequality}, it follows that $\left(\int|\widehat{g}\ast\chi_{k,\lambda}|^p\,\mathrm{d}\mu\right)^{\frac{1}{p}}
\lesssim
b^{\frac{k}{r}}\|g\|_r$. By duality, for all $h\in L^{p^\prime}(\mathrm{d}\mu)$,
\begin{equation*}
	\left|\int\left(\widehat{g}\ast\chi_{k,\lambda}\right)(x)\cdot h(x)\,\mathrm{d}\mu(x)\right|
	\lesssim
	b^{\frac{k}{r}}\|g\|_r\|h\|_{L^{p^\prime}(\mathrm{d}\mu)}.
\end{equation*}
By Fubini's theorem,
\begin{equation*}
	\left|\int_{(\Delta^\ast_k)^c}\widehat{g}(y)\int\chi_{k,\lambda}(x-y) h(x)\,\mathrm{d}\mu(x)\,\mathrm{d}y\right|
	\lesssim
	b^{\frac{k}{r}}\|g\|_r\|h\|_{L^{p^\prime}(\mathrm{d}\mu)}.
\end{equation*}
Now we choose $r=2/(2-p)\geq2$, since $1\leq p<2$. Hausdorff--Young inequality says that $\|g\|_r\leq \|\widehat{g}\|_{r^\prime}$, so
\begin{equation*}
	\left|\int_{(\Delta^\ast_k)^c}\widehat{g}(y)\int\chi_{k,\lambda}(x-y) h(x)\,\mathrm{d}\mu(x)\,\mathrm{d}y\right|
	\lesssim
	b^{\frac{k}{r}}
	\|\widehat{g}\|_{r^\prime}
	\|h\|_{L^{p^\prime}(\mathrm{d}\mu)}.
\end{equation*}
By duality,
\begin{equation*}
	\left\|\int\chi_{k,\lambda}(x-y) h(x)\,\mathrm{d}\mu(x)\right\|_{L^r\left((\Delta^\ast_k)^c,\mathrm{d}y\right)}
	\lesssim
	b^{\frac{k}{r}}
	\|h\|_{L^{p^\prime}(\mathrm{d}\mu)},
\end{equation*}
and again by duality,
\begin{equation*}
	\left\|
	\left\|
	\chi_{k,\lambda}(x-y)
	\right\|_{L^p\left(\R^d,\mathrm{d}\mu(x)\right)} \right\|_{L^r\left((\Delta^\ast_k)^c,\mathrm{d}y\right)}
	\lesssim
	b^{\frac{k}{r}}.
\end{equation*}
Therefore,
\begin{equation*}
	\begin{aligned}
		\int_{(\Delta^\ast_k)^c}\bigg|\int\chi_{k,\lambda}(x-y)\,\mathrm{d}\mu(x)\bigg|^\frac{r}{p}\,\mathrm{d}y
		\lesssim
		b^k.
	\end{aligned}
\end{equation*}
The rest part of the proof is similar to that of $p=1$.
\subsection{Proof of Theorem~\ref{The H^1 to L^p boundedness criterion}, when $2\leq p<\infty$} This is somewhat different from the case $1\leq p<2$. When $d=1$ (hence isotropic) and $p=2$, Odysseas Bakas gives a proof in \cite{Bakas}, Proposition~8, which works well under the anisotropic settings after minor modifications. The proof of this part is essentially due to Bakas.
 
We begin with a square function characterization of $H^1$. Let $\eta$ be a Schwartz function on $\R^d$ whose Fourier transform is nonnegative, supported in $(A^\ast)^2(\Delta^\ast)\setminus(A^\ast)^{-1}(\Delta^\ast)$, equal to $1$ on $(A^\ast)(\Delta^\ast)\setminus\Delta^\ast$. We define $\Psi$ via the Fourier transform,
\begin{equation*}
	\widehat{\Psi}(\xi)
	\coloneqq
	\frac{\eta(\xi)}{\sum_{j\in\Z}\eta\left((A^\ast)^{-j}\xi\right)}.
\end{equation*}
It is easy to verify that $\widehat{\Psi}\in C^\infty_{c}(\R^d)$, and $\widehat{\Psi}(\xi)\gtrsim1$ on $(A^\ast)(\Delta^\ast)\setminus\Delta^\ast$, moreover,
\begin{equation*}
	\sum_{j\in\Z}\widehat{\Psi}\left((A^\ast)^{-j}\xi\right)
	=1,
	\text{ when }\xi\neq0.
\end{equation*}
Simple calculation shows that $\widehat{\Psi}\left((A^\ast)^{-j}\xi\right)=\widehat{\Psi_j}(\xi)$, where $\Psi_j(x)\coloneqq b^j\Psi(A^jx)$.  The following theorem is proved in \cite{Liu--Yang--Yuan}, Theorem~2.8.
\begin{thm_no_numbering}[Square function characterization of anisotropic $H^1$]\label{Square function characterization of H^1}
	Assume that $f$ is a tempered distribution. The square function (or the g-function) of $f$ is defined by
	\begin{equation*}
		g(f)(x)
		\coloneqq
		\left(\sum_{j\in\Z}\left|\Psi_j\ast f\right|^2(x)\right)^{\frac{1}{2}}.
	\end{equation*}
If $f\in H^1$, then $\|f\|_{H^1}\approx\|g(f)\|_{L^1}$.
\end{thm_no_numbering}
With these preparations in place, we can begin the proof.
\paragraph{\textit{The sufficient part}}Assume that \eqref{Condition: control of the ell^{2/(2-p)} sum over rectangels} holds, we decompose the integral into the sum
\begin{equation*}
	\int_{\R^d}|\widehat{f}(\xi)|^p\,\mathrm{d}\mu(\xi)
	=
	\sum_{l\in\Z}	\int_{(A^\ast)^l(\Delta^\ast)\setminus(A^\ast)^{l-1}(\Delta^\ast)}|\widehat{f}(\xi)|^p\,\mathrm{d}\mu(\xi).
\end{equation*}
We fix an $l\in\Z$ and assume that $\xi\in(A^\ast)^l(\Delta^\ast)\setminus(A^\ast)^{l-1}(\Delta^\ast)$, which is equivalent to $\rho_\ast(\xi)=b^{l-1}$. It is not difficult to see that the sum $\sum_{j\in\Z}\widehat{\Psi}\left((A^\ast)^{-j}\xi\right)$ is 
\begin{equation*}
\sum_{l-2\leq j\leq l+2}\widehat{\Psi}\left((A^\ast)^{-j}\xi\right),
\end{equation*}
which has only 5 terms. So we can write $\widehat{f}(\xi)$ as
\begin{equation*}
	\widehat{f}(\xi)\cdot 1
	=
	\widehat{f}(\xi)\sum_{l-2\leq j\leq l+2}\widehat{\Psi}\left((A^\ast)^{-j}\xi\right)
	=
	\sum_{l-2\leq j\leq l+2}\widehat{\Psi_j\ast f}(\xi).
\end{equation*}
Taking supremum over $\xi$,
\begin{equation*}
	\sup_{\xi\in(A^\ast)^l(\Delta^\ast)\setminus(A^\ast)^{l-1}(\Delta^\ast)}\left|\widehat{f}(\xi)\right|
	\leq
	\sum_{l-2\leq j\leq l+2}\left\|\widehat{\Psi_j\ast f}\right\|_\infty
	\leq
	\sum_{l-2\leq j\leq l+2}\left\|\Psi_j\ast f\right\|_1.
\end{equation*}
Integrating over $(A^\ast)^l(\Delta^\ast)\setminus(A^\ast)^{l-1}(\Delta^\ast)$ shows that
\begin{equation*}
	\begin{aligned}
		\int_{(A^\ast)^l(\Delta^\ast)\setminus(A^\ast)^{l-1}(\Delta^\ast)}|\widehat{f}|^p\,\mathrm{d}\mu
		&\leq
		\mu\left((A^\ast)^l(\Delta^\ast)\setminus(A^\ast)^{l-1}(\Delta^\ast)\right)
		\left(\sum_{l-2\leq j\leq l+2}\left\|\Psi_j\ast f\right\|_1\right)^p\\
		&\lesssim
		\sum_{l-2\leq j\leq l+2}\left\|\Psi_j\ast f\right\|_1^p.
	\end{aligned}
\end{equation*}
where we used the hypothesis \eqref{Condition: control of the ell^{infty} sum over annulus} and the trivial estimate $\left(\sum_{i=1}^5|a_i|\right)^p\lesssim\sum_{i=1}^5|a_i|^p$. Summing over $l\in\Z$, and use the condition $p\geq2$,
\begin{equation*}
	\begin{aligned}
		\left(\int|\widehat{f}|^p\,\mathrm{d}\mu\right)^\frac{1}{p}
		\lesssim
		\left(\sum_{j\in\Z}\left\|\Psi_j\ast f\right\|_1^p\right)^\frac{1}{p}
		\lesssim
		\left(\sum_{j\in\Z}\left\|\Psi_j\ast f\right\|_1^2\right)^\frac{1}{2}.
	\end{aligned}
\end{equation*}
Finally, the integral form of Minkowski's inequality yields
\begin{equation*}	\left(\sum_{j\in\Z}\left\|\Psi_j\ast f\right\|_1^2\right)^\frac{1}{2}
\leq
\left\|\left(\sum_{j\in\Z}\left|\Psi_j\ast f\right|^2\right)^{\frac{1}{2}}\right\|_1
\approx
\|f\|_{H^1}.
\end{equation*}
\paragraph{\textit{The necessary part}}Assume that \eqref{Condition: the H^1 to F(L^p) inequality} holds. For $l\in\Z$, we choose $f\coloneqq \Psi_{-l}$. A direct calculation shows that
\begin{equation*}
	\begin{aligned}
		\|f\|_{H^1}
		\approx
		\Big\|\Big(\sum_{-l-2\leq j\leq -l+2}|\Psi_j\ast f|^2\Big)^\frac{1}{2}\Big\|_{1}
		\approx
		\sum_{-l-2\leq j\leq -l+2}\|\Psi_j\ast f\|_{1}.
	\end{aligned}
\end{equation*}
After a change of variable, the last sum is
\begin{equation*}
		\sum_{j=-2}^2\int\left|\int\Psi(z)\Psi(x-A^jz)\,\mathrm{d}z\right|\,\mathrm{d}x
		\approx
		1.
\end{equation*}
Therefore, the $H^1$ norm of $\Psi_{-l}$ is independent of $l$. From the construction of $\Psi$, 
\begin{equation*}
	\widehat{\Psi_{-l}}(\xi)\gtrsim1 \text{ on } (A^\ast)^{-l+1}(\Delta^\ast)\setminus(A^\ast)^{-l}(\Delta^\ast).
\end{equation*}
So for $l\in\Z$,
\begin{equation*}
	\mu\left((A^\ast)^{-l+1}(\Delta^\ast)\setminus(A^\ast)^{-l}(\Delta^\ast)\right)^{\frac{1}{p}}
	\lesssim
	\left(\int_{\R^d}|\widehat{\Psi_{-l}}(\xi)|^p\,\mathrm{d}\mu(\xi)\right)^{\frac{1}{p}}
	\lesssim
	\|\Psi_{-l}\|_{H^1(\R^d)}
	\approx
	1,
\end{equation*}
where all implicit constants in $\lesssim$ are independent of $l$. Consequently, the inequality \eqref{Condition: control of the ell^{infty} sum over annulus} holds.

\printbibliography
\end{document}